\documentclass[12pt]{article}
\usepackage{latexsym}
\usepackage{amssymb, amsmath, xspace, lscape,  latexsym}
\usepackage{amsthm}
\usepackage{color}

\newcommand{\nn}{\nonumber}
\newcommand{\bea}{\begin{eqnarray}}
\newcommand{\eea}{\end{eqnarray}}
\def\beq#1#2\eeq{
        \begin{equation}
        \label{#1}
            #2
        \end{equation}}

\newcommand{\cH}{{\cal H}}

\newcommand{\al}{\alpha}

\def\btheor#1\etheor{
        \begin{theor}
            #1
        \end{theor}
    }

    \def\bsled#1\esled{
        \begin{sled}
            #1
        \end{sled}   }

\newtheorem{theorem}{Theorem}

\newtheorem{lemma}{Lemma}

\newtheorem{cor}{Corollary}

\newtheorem{remark}{Remark}

\def\hm#1{#1\nobreak\discretionary{}{\hbox{\m@th$#1$}}{}}
\def\mi#1{\discretionary{\hbox{\m@th$#1$}}{\hbox{\m@th$#1$}}{}}

\vspace{4ex}
\begin{document}
\title{\bf Singular linear statistics of the Laguerre Unitary Ensemble and Painlev\'e III: Double scaling analysis.}
\author{Min Chen\thanks{chenminfst@gmail.com} and Yang Chen\thanks{Corresponding author: yayangchen@umac.mo, yangbrookchen@yahoo.co.uk and Telephone:(853)8822-8546.}\\
        Department of Mathematics, University of Macau,\\
        Avenida da Universidade, Taipa, Macau, China.
        }
\date{}
\maketitle
\begin{abstract}
We continue with the study of the Hankel determinant, defined by,\\
${\small
D_{n}(t,\alpha)=\det\left(\int_{0}^{\infty}x^{j+k}w(x;t,\alpha)dx\right)_{j,k=0 }^{n-1},}
$
generated by a singularly perturbed Laguerre weight,
$
w(x;t,\alpha)=x^{\alpha}{\rm e}^{-x}{\rm e}^{-t/x},x\in \mathbb{R}^+,\;\alpha>0,\;t>0,
$
obtained through a deformation of the Laguerre weight function,
$
w(x;0,\alpha)=x^{\alpha}{\rm e}^{-x},\;x\in\mathbb{R}^+,\; \alpha>0,
$
via the multiplicative factor ${\rm e}^{-t/x}$.
\\
An earlier investigation was made on the finite $n$ aspect of such determinants, which appeared in \cite{ci1}. It was found
that the logarithm of the Hankel determinant has an integral representation in terms of a particular Painlev\'{e} {\uppercase\expandafter{\romannumeral3}}(${\rm P_{III}},$ for short)
and its $t$ derivatives. In this paper we show that, under a double scaling, where $n$, the order of the
Hankel matrix tends to $\infty,$ and $t$, tends to $0^+$, the scaled---and therefore, in some sense, infinite
dimensional---Hankel determinant, has an integral representation in terms of a $C$ potential. The second order non-linear
ode satisfied by $C,$ after a change of variable, is another ${\rm P}_{III}$ transcendent, albeit with fewer number of parameters.
\\
Expansions of the double scaled determinant for small and large parameter
are obtained.

\end{abstract}

\vfill\eject

\setcounter{equation}{0}
\section{Introduction.}

The analysis of Hankel determinants plays an important role in random matrix theory \cite{M2004}.
The second author and his collaborators made use of a theorem in linear statistics to
study Hankel determinants and the associated orthogonal polynomials,
see \cite{BY2005, 1BCW2001, 2BCW2001, YN1998, YM1994}
based on Dyson's Coulomb Fluid \cite{Dyson1962}. See also \cite{YIsmail1997, YIsmail1994}. Bonan, Nevai and others
studied orthogonal polynomials with ladder operators.
Works on the Hankel determinants generated by the deformations of classical weight, maybe found in \cite{BAU1990, SA1994, SDP1987, SP1984, BC1986, BC1990},
and also in \cite{BY2009, YM2006, YIS1997, YIS2005, YG2005, YZ2010}.
 \\
Adler and Van Moerbeke obtained differential equations governing the logarithmic derivatives of Hankel determinants, and their generalization,
using a multi-time approach, \cite{adlervan}. The connections between the Hankel determinants generated by quite general semi-classical weights are established in \cite{B2009} and \cite{BEH2006}. The asymptotics of $n$ dimensional Toeplitz determinants with Fisher-Hartwig
 singularities are studied in
 \cite{DIK2011}, and in a review, \cite{K2011}.
\par
The joint probability density function of the eigenvalues, $\{x_{j}:j=1,..., n\}$ for an ensemble of
$n\times{n}$ Hermitian matrix is given by \cite{M2004}
\begin{equation*}
p(x_{1},x_{2},\ldots,x_{n})\prod_{j=1}^{n}dx_{j}=\frac{1}{D_{n}[w]}\:\frac{1}{n!}\:
\prod_{1\leq{j}<k\leq{n}}\left(x_{j}-x_{k}\right)^{2}\prod_{\ell=1}^{n}w(x_{\ell})dx_{\ell},
\end{equation*}
where $D_{n}[w]$ is the normalization constant, or the partition function,
\begin{equation*}
D_{n}[w]=\frac{1}{n!}\int_{\mathbb{R}_{+}^{n}}\prod_{1\leq{j}<k\leq{n}}\left(x_{j}-x_{k}\right)^{2}
\prod_{\ell=1}^{n}w(x_{\ell})dx_{\ell}.
\end{equation*}
Here $w(x)$ is a positive weight function supported on $\mathbb{R}_{+},$ with moments given by,
$$
\mu_j[w]:=\int_{\mathbb{R}_{+}}x^{j}w(x)dx,\;\;\;j\in\{0,1,2,...\}.
$$
It is a well-known fact that $D_n[w]=\det(\mu_{j+k}[w])_{0\leq j,k\leq n-1},$ and this  gives a link between the determinant
of the Hankel matrix $(\mu_{j+k})$ and the partition function $D_{n}[w]$ given by the multiple integral.
\par
A linear statistics is defined to be the sum of a function $f(x)$ evaluated at the random  variables
$\{x_{j}\in{\mathbb{R}_{+}}:j=1,\ldots,n\},$ namely,
$$\sum_{j=1}^{n}f(x_{j}).$$
\\
In our problem, $f(x)>0,\;\;x\in\mathbb{R}_{+}.$
\\
Denote by $M_f(t)$ the moment generating of the linear statistics. This is obtained by a Laplace transform of the probability density function,
$\mathbb{P}_f(Q),$
\begin{equation*}
M_{f}(t)=\int_{\mathbb{R}_{+}}\mathbb{P}_{f}(Q)e^{-tQ}dQ
=\frac{{\rm det}\left(\mu_{j+k}(t)\right)_{j,k=0}^{n-1}}{{\rm \det}\left(\mu_{j+k}(0)\right)_{j,k=0}^{n-1}},
\end{equation*}
where
\begin{equation*}
\mu_{j}(t):=\int_{\mathbb{R}_{+}}x^{j}w(x)e^{-tf(x)}dx, \qquad j\in\{0,1,\ldots\}.
\end{equation*}
Hence the moment generating function becomes the quotient of Hankel determinants.
\par
For the problem at hand, the finite $n$ version of which first appeared in \cite{ci1}, where $f(x):=1/x$, and
$w(x)$ is the Laguerre weight supported on $\mathbb{R}_{+}:$
$$
w(x)=x^{\alpha}{\rm e}^{-x},\;\;\alpha > 0.
$$
Such a linear statistics leads to the perturbed Laguerre weight,
\begin{equation}\label{weight}
 w(x; t,\alpha)=w(x){\rm e}^{-\frac{t}{x}}=x^{\alpha}{\rm e}^{-x -t/x}, \quad x\in \mathbb{R}_{+}, \quad t\geq 0, \quad \alpha>0,
\end{equation}
and the perturbed Hankel determinant:
\begin{equation}\label{det}
D_{n}(t,\alpha) =\det\left(\int_{\mathbb{R}_{+}}x^{j+k}w(x;t,\alpha)dx\right)_{j,k =0}^{n-1}.
\end{equation}
Such a perturbation is motivated in part by a finite temperature integrable quantum field theory \cite{luk}.
The Hankel determinant given by (\ref{det}) maybe interpreted  as the generating function for the distribution function of the
Wigner delay time in chaotic cavities and recently studied in \cite{TM2013} through large deviation techniques.
Mathematically, $f(x)=1/x,$  introduces an infinitely strong zero on the weight function
at 0, and has the effect of pushing the left end point of the equilibrium density,
(or charge density, if we view the eigenvalues as charges), away from $0$ at a slow speed. If $a$ is left end point of the support
of the density, then for $t>0,$ and large $n$, $a={\rm O}\left(\frac{t^{2/3}}{n^{1/3}}\right).$ See the top equation of page 16.
\\
For the Laguerre weight, $x^{\alpha}{\rm e}^{-x},\;\alpha>0,\;x\in\mathbb{R}_{+},$
$a={\rm O}\left(\frac{\alpha^2}{n}\right).$
\par
One finds, in applications, other forms of $f(x)$. For example, the characterization of Shannon capacity \cite{YM2012} in the study of
the outage and error probability in wireless communication. In this situation, the Laguerre weight is multiplied
by $\exp(-\lambda\ln(1+x/t))$. Here $\lambda$ is the parameter which ``generates" the Shannon capacity and $t (>0) $ plays the role
of the time variable in the ensuing Painleve V equation. See Basor and Chen \cite{BY2014}, for a recent review of this and other related matters.
A similar multiplicative factor on the Jacobi weight,
$x^{\alpha}(1-x)^{\beta},\;0\leq x\leq 1$ for $\alpha=\pm\;1/2,$ and $\beta=\pm\:1/2,$ arose in enumeration problems related to the moduli space of super-symmetric QCD in the Veneziano limit \cite{CJJM2013}.
The Coulomb Fluid interpretation of the eigenvalues has been adopted to compute statistical properties
involving large deviations \cite{MS2014}.
\par
\par
In this paper we shall be concerned with a double scaling scheme, where $t\to 0^+,$ and $n\to\infty$, such that $s:=(2n+1+\alpha)t$ is finite. And ultimately provide a description of the (double-scaled) Hankel determinant.
\par

The remainder of this paper is organized as follows. Section 2 recalls elementary\\
\noindent facts about orthogonal polynomials,
a description of the Dyson Coulomb Fluid, followed by recalling certain results obtained in Chen and Its \cite{ci1}.
We show that that double-scaled and in some sense infinite dimensional Hankel determinant, has an integral
representation in terms of a $C-$ potnetial, which satisfies a second order non-linear ode in $s$, which is equivalent
to a ``lesser" ${\rm P_{III}}.$   The $\sigma$ function of Jimbo-Miwa-Okamoto associated with
this particular ${\rm P_{III}}$ is also found. In section 3,  we compute formal power series for the small $s$
and large $s$  behavior of the $C(s)$. In section 4, an evaluation is made on the constant term of the monic polynomials
orthogonal with respect to $w(x;t,\alpha),$ namely,
$P_n(0;t,\alpha),$ for large $n$ and $s=2nt.$ We obtained relationship between certain constants, and
from which the value of the constant of the asymptotic expansion of the double-scaled Hankel determinant
is conjectured. We conclude in section 5.
\\
\section{Double Scaling.}
In this section we study the effect of double scaling by sending $n\to\infty$ and $t\to 0^+$ and such that $s:=(2n+1+\alpha)t$ is fixed.
\par
To set the stage for later development, we recall for the Reader that the Hankel determinant can also be expressed as
\begin{equation}\label{hn}
D_n(t,\alpha) = \prod_{j=0}^{n-1}h_j(t),
\end{equation}
where $\{h_j(t): j=0,...,n-1\}$ are the squares of the $L^2$ norm of the monic
polynomials $P_n(x)$ orthogonal with respect to $w(x;t,\alpha)$, namely,
\begin{equation}\label{poly}
\int_{\mathbb{R}_{+}}P_{j}(x)P_{k}(x)w(x;t,\alpha)dx = h_j(t)\delta_{jk}.
\end{equation}
The monic polynomials satisfy the recurrence relations,
$$
xP_{n}(x)=P_{n+1}(x)+\alpha_{n}(t)P_{n}(x)+\beta_{n}(t)P_{n-1}(x),
$$
subject to the initial data, $P_{0}(x)=1$, and $\beta_{0}P_{-1}(x)=0$. It is clear that
the coefficients of the polynomials $P_{n}(x)$ and the recurrence coefficients $\alpha_{n}$,
$\beta_{n}$ all depend on $t$. We shall see later that $\alpha_n(t)$, the diagonal recurrence coefficient plays an important role.
\par
Heine's formula, gives the multiple integral representation of the orthogonal polynomials,
\bea
P_n(z;t,\alpha)=\frac{1}{D_n(t,\alpha)}\frac{1}{n!}\int_{\mathbb{R}_{+}^{n}}\prod_{m=1}^{n}(z-x_m)
\:\prod_{1\leq j<k \leq n}(x_{j}-x_{k})^2\prod_{\ell=1}^{n}w(x_{\ell};t,\alpha)dx_{\ell}, \nonumber
\eea
 see Szeg{\"o} \cite{Szego1939}.
Here, the Hankel determinant reads,
$$
D_n(t,\alpha)=\det\left(\int_{\mathbb{R}_{+}}x^{i+j}w(x;t,\alpha)dx\right)_{0\leq i,j\leq n-1},
$$
$$
\quad\quad=\frac{1}{n!}\int_{\mathbb{R}_{+}^{n}}\prod_{1\leq j<k \leq n}(x_{j}-x_{k})^2
\prod_{\ell=1}^{n}w(x_{\ell};t,\alpha)dx_{\ell}. \nonumber
$$
\\
It is clear that,
\begin{equation}\label{alpha}
(-1)^nP_n(0; t,\alpha) = \frac{D_{n}(t, \alpha + 1)}{D_{n}(t, \alpha)},
\end{equation}
and for the {\it monic} Laguerre polynomials, namely, $P_n(z;0,\alpha),$ there is a closed form evaluation at $z=0$, see ((5.1.7), \cite{Szego1939}),
\begin{equation}\label{gamma}
(-1)^nP_n(0; 0,\alpha)=\frac{\Gamma(n+1+\al)}{\Gamma(1+\al)},\;\;\alpha>-1.
\end{equation}
\\
We give a brief description of Coulomb Fluid.
The energy of a system of $n$ logarithmically repelling particles on the line, confined by an external potential $\rm{v}$ reads
$$
E(x_1,x_2,...,x_n)=-2\sum_{1\leq j<k\leq n}\ln|x_j-x_k|+\sum_{j=1}^{n}{\rm v}(x_j).
$$
For sufficiently large $n$, the particles may be approximated by a continuous fluid\cite{Dyson1962}, with a density $\sigma$.
In the Coulomb fluid approximation, $\sigma(x),$ assumed to be supported on $[a,b]$ is obtained by minimizing the free-energy functional,
$F[\sigma]$,
\begin{equation*}
\min\limits_{\sigma>0}F[\sigma] \quad {\rm subject \quad to} \quad \int_{a}^{b}\sigma(x)dx=n,
\end{equation*}
where
\begin{equation*}
F[\sigma]:=\int_{a}^{b}\sigma(x){\rm v}(x)dx-\int_{a}^{b}\int_{a}^{b}\sigma(x)\ln|x-y|\sigma(y)dxdy.
\end{equation*}
\par
Upon minimization, the density $\sigma(x)$ is found to satisfy the integral equation,
\begin{equation}\label{1rr}
A={\rm v}(x)-2\int_{a}^{b}\ln|x-y|\sigma(y)dy, \qquad x\in [a,b],
\end{equation}
where $A$, the Lagrange multiplier imposes the constraint $\int_{a}^{b}\sigma(x)dx=n$. For more information, see \cite{YIsmail1997} and
references therein. Note that $A$ is a constant independent of $x$ for $x\in[a,b]$, but $A$ and $\sigma$ depend on $t$ and $n$.
The relations between Coulomb fluid and orthogonal polynomials, where the potential ${\rm v}$, being convex,
can be found, for example, in \cite{YIsmail1997}. A description of the density for exponential
weights and strong asymptotics for the  orthogonal polynomials can be found in \cite{DKMVZ1999}. For further information on the equilibrium density see \cite{Defit}.
\par
For the problem at hand, the density, supported on $[a,b]$ reads,
$$
\sigma(x)=\frac{\sqrt{(b-x)(x-a)}}{2\pi}\left[\left(\frac{\alpha}{\sqrt{ab}}+\frac{t(a+b)}{2(ab)^{\frac{3}{2}}}\right)
\frac{1}{x}+\frac{t}{x^{2}\sqrt{ab}}\right], \quad a \leq x \leq b.
$$
This is obtained by solving a singular integral equation, found by taking a derivative with respect to $x$ on (\ref{1rr});
$$
{\rm v}'(x)-2P\int_{a}^{b}\frac{\sigma(y)}{x-y}dy=0, \quad x\in[a,b].$$
Here $P$ denotes the Cauchy principal value,  with the condition that the equilibrium density, $\sigma$, vanishes
at the end points of the support.
\\
See \cite{YIsmail1997, YN1998}, for further discussion.
For this problem ${\rm v}(x)=-\ln w(x; t,\alpha),$ and $${\rm v}'(x)=-\frac{\alpha}{x}+1-\frac{t}{x^2}.$$
\\
The end points of the interval $[a,b]$ are determined by the normalization condition
$$
\int_{a}^{b} \sigma(x)dx =n,
$$
and a supplementary conditions,  which can be found for example, in \cite{YIsmail1997} and \cite{YN1998}. These are
\bea
\int_{a}^{b}\frac{x{\rm v}'(x)}{\sqrt{(b-x)(x-a)}}dx=2\pi\:n, \nonumber
\eea
and
\bea
\int_{a}^{b}\frac{{\rm v}'(x)}{\sqrt{(b-x)(x-a)}}dx=0. \nonumber
\eea
With the aid of the integrals in Appendix $A$, the end points of the support, $a$ and $b$, satisfy algebraic equations,
\bea\label{c11}
2n+\alpha+\frac{t}{\sqrt{ab}}=\frac{a+b}{2},
\eea
and
\bea\label{c12}
\frac{(a+b)t}{2(ab)^{\frac{3}{2}}}+\frac{\alpha}{\sqrt{ab}}=1.
\eea
We state a lemma here, which gives further insight into a particular
\\
${\rm P_{\uppercase\expandafter{\romannumeral3}'}}(-4(2n+1+\alpha), -4\alpha, 4, -4),$
that appeared in the finite $n$ setting.
\begin{lemma}
The geometric mean of  $a$ and $b,$ namely, $\widetilde{X}:=\sqrt{ab}$
satisfies the quartic equation,
\bea\label{c13}
\widetilde{X}^4-\al\:\widetilde{X}^3-(2n+\al)t\:\widetilde{X}-t^2=0.
\eea
\end{lemma}

\begin{proof}
The equation $(\ref{c13})$ is found by eliminating $\frac{a+b}{2}$ from $(\ref{c11})$ and $(\ref{c12}).$
\end{proof}
In order to characterize the large $n$ behavior of the Hankel determinant,
we recall certain results in \cite{ci1}. For convenience, we use $t$ and $y_n(t)$ instead of $s$ and $a_n(s)$ adopted in\cite{ci1}.
See (Theorem 1, \cite{ci1}).
\begin{theorem}
The diagonal recurrence coefficients, $\al_n(t),$ maybe expressed as,
$$
\al_n(t)=2n+1+\alpha+y_n(t),
$$
where the auxiliary quantity $y_{n}(t),\; n=0,1,2,...$ satisfies
\bea\label{c2}
y''_{n}=\frac{(y'_{n})^2}{y_{n}}-
\frac{y'_{n}}{t}+(2n+1+\alpha)\frac{y_{n}^{2}}{t^{2}}+\frac{y_{n}^{3}}{t^{2}}+\frac{\alpha}{t}-\frac{1}{y_{n}},
\eea
with the initial conditions
\bea\label{c3}
y_n(0)=0, \quad y'_{n}(0)=\frac{1}{\alpha}, \quad \alpha>0.
\eea
If $y_{n}(t)$:=-q(t), then $q(t)$ is a solution of ${\rm P_{\uppercase\expandafter{\romannumeral3}'}}(-4(2n+1+\alpha), -4\alpha, 4, -4)$,
 following the convention of \cite{KSO2006}.
\end{theorem}

Substituting
\bea
y_{n}(t):=\frac{t}{X_{n}(t)}, \nonumber
\eea
into $(\ref{c2})$ it is seen that $X_n(t),$ satisfies,
\bea\label{c5}
X''_{n}=\frac{(X'_{n})^2}{X_{n}}-\frac{X'_{n}}{t}-\frac{\alpha{(X_{n})^{2}}}{t^{2}}-\frac{2n+1+\alpha}{t}
+\frac{X_{n}^{3}}{t^{2}}-\frac{1}{X_{n}},
\eea
with the boundary condition,
$$
X_{n}(0)=\alpha, \quad \alpha>0,
$$
which is recognized to be an equivalent ${\rm P_{\uppercase\expandafter{\romannumeral3}'}}(-4\alpha, -4(2n+1+\alpha), 4, -4)$.
\\
If we disregard the derivatives in $(\ref{c5})$, then $X_{n}$ satisfies a quartic
equation,
\bea
X_n^{4}-\alpha{X_n^{3}}-(2n+1+\alpha)t\:X_n-t^{2}=0. \nonumber
\eea
Note that the quartic obtained above becomes the quartic of $(\ref{c13})$, if we
replace $2n+1$ by $2n$.
Theorem 2 in \cite{ci1} reveals an important relation between $X_{n}$ and the logarithmic derivative of the Hankel determinant
which we recall in the following Theorem.
\begin{theorem}
\begin{align}\label{c7}
\ln\frac{D_{n}(t,\alpha)}{D_{n}(0,\alpha)}
=\int_{0}^{t}\left(\frac{\xi}{2}-\frac{1}{4}\left(X_{n}-\alpha\right)^{2}-(n+\frac{\alpha}{2})
\frac{\xi}{X_{n}}-\frac{\xi^{2}}{4X_{n}^{2}}+\frac{\xi^{2}(X'_{n})^{2}}{4X_{n}^{2}}\right)\frac{d\xi}{\xi},
\end{align}
where $D_{n}(0,\alpha)$ has a closed form evaluation;
\begin{equation*}
D_{n}(0,\alpha)=\frac{G(n+1)G(n+\alpha+1)}{G(\alpha+1)}.
\end{equation*}
Here $G(z),\:z\in \mathbb{C}\cup\{\infty\}$ is the Barnes $G$-function, an entire function of order $2$, and  satisfies the functional relation $G(z+1)=\Gamma(z)\:G(z),$
where $G(1)=1$.
\end{theorem}
The next Theorem recalls equations (3.23) and (3.24) in \cite{ci1}.
\begin{theorem}
If
\bea\label{c8}
H_{n}(t):=t\frac{d}{dt}\ln{\frac{D_{n}(t,\alpha)}{D_{n}(0,\alpha)}},
\eea
then
\bea\label{c9}
(tH_{n}'')^{2}=\left[n-(2n+\alpha)H_{n}'\right]^{2}-4\left[n(n+\alpha)+tH_{n}'-H_{n}\right]H_{n}'(H_{n}'-1),
\eea
subject to the initial condition, $H_{n}(0)=0.$
\end{theorem}
The equation $(\ref{c9})$ is the Jimbo-Miwa-Okamoto $\sigma$-from of ${\rm P_{\uppercase\expandafter{\romannumeral3}}},$
see \cite{JM1982}. From $(\ref{c7})$ and $(\ref{c8})$, we see that
$H_{n}(t)$ maybe expressed in terms of $X_n(t)$ as
\bea\label{c10}
H_{n}=\frac{t}{2}-\frac{1}{4}\left(X_{n}-\alpha\right)^{2}-(n+\frac{\alpha}{2})\frac{t}{X_{n}}-\frac{t^{2}}{4X_{n}^{2}}
+\frac{t^{2}(X'_{n})^{2}}{4X_{n}^{2}}.
\eea
Recall the double scaling process: Sending $n\to\infty$ and $t\to 0^+$ in such a way that $s:=(2n+1+\alpha)t$ is fixed.\\
In the next theorem, we find that it is convenient to introduce a ``potential'' defined by
\begin{equation}\label{0c15}
C(s)=\lim_{n\to\infty}\frac{y_{n}\left(\frac{s}{2n+1+\al}\right)}{\frac{s}{2n+1+\al}}
=\lim_{n\to\infty}\frac{1}{X_n\left(\frac{s}{2n+1+\al}
\right)}.
\end{equation}

\begin{theorem}
Let
\bea
s:=(2n+1+\al)t, \nonumber
\eea
where $t\to 0^+$ and $2n+1+\al\to\infty$, and such that $s$ is finite. Let
\bea\label{c15}
\Delta(s,\al):=\lim_{n\to\infty}\frac{D_n\left(\frac{s}{2n+1+\al},\alpha\right)}{D_n\left(0,\alpha\right)}, \quad  {\rm where}
 \quad \Delta(0,\al)=1.
\eea
The $C$ potential,  satisfies
\bea\label{c16}
C''(s)=\frac{(C'(s))^2}{C}-\frac{C'(s)}{s}+\frac{(C(s))^2}{s}+\frac{\alpha}{s^2}-\frac{1}{s^2\:C(s)},
\eea
with the initial condition
$
C(0)=1/\alpha.
$
If
\bea
{\cal H}(s):=\lim_{n\to\infty}H_{n}(s/(2n+1+\alpha))=s\frac{d}{ds}\ln\Delta(s,\al), \nonumber
\eea
then ${\cal H}(s)$ satisfies
\bea\label{c18}
(s\cH'')^2+4(\cH')^2\left(s\:\cH'-\cH\right)-\left(\al\:\cH'+\frac{1}{2}\right)^{2}=0,
\eea
subject to the initial condition, ${\cal H}(0)=0.$
Moreover,
\bea\label{c19}
{\cal H}(s)=\frac{1}{4}\left(\frac{sC'(s)}{C(s)}\right)^2-\frac{sC(s)}{2}-\frac{1}{4}\left(\frac{1}{C(s)}-\alpha\right)^2.
\eea
\end{theorem}
\begin{proof}
By a straight forward, formal, if tedious computations, we see that (\ref{c5}) becomes (\ref{c16}),
 (\ref{c9}) becomes (\ref{c18}), and (\ref{c10}) becomes (\ref{c19}), after double scaling.
\end{proof}

\begin{remark}
Form (\ref{c7}) and (\ref{c15}), we have,
\bea\label{cc1}
\ln{\Delta(s,\alpha)}=\int_{0}^{s}\left\{\frac{1}{4}\left(\frac{\xi}{C(\xi)}\:\frac{dC(\xi)}{d\xi}\right)^2-
\frac{\xi\:C(\xi)}{2}-\frac{1}{4}\left(\frac{1}{C(\xi)}-\al\right)^2\right\}\frac{d\xi}{\xi}.
\eea
\end{remark}
\begin{remark}By a change of variables $Y(x)=\frac{x}{2}C(\frac{x^{2}}{8}),$ $x\in(0,\infty),$ the equation $(\ref{c16})$, becomes
a ``lesser" $P_{III};$
\bea\label{c21}
Y''=\frac{(Y')^{2}}{Y}-\frac{Y'}{x}+\frac{Y^{2}}{x}-\frac{1}{Y}+\frac{2\alpha}{x}.
\eea
\\
From \cite{KSO2006}, we see that $(\ref{c21})$ is ${\rm P_{III}(1, 2\alpha, 0, -1)}$;
 the ${\rm P_{III}}$ with a fewer number of parameters, mentioned in the abstract.
\end{remark}
\begin{remark}
Substituting
$$
\mathbb{H}(s)=:{\cal H}(2s)-\frac{\alpha^{2}}{4},
$$
into $(\ref{c18})$, we find,
\bea
\left(s\mathbb{H}''(s)\right)^{2}+4\mathbb{H}'^{2}(s)\left(s\mathbb{H}'(s)-\mathbb{H}(s)\right)-2\alpha\:{\mathbb{H}'(s)}-1=0, \nonumber
\eea
which is the ${\rm P}_{III}$ obtained by Ohyama-Kawamuko-Sakai-Okamoto, see ($(18)$, in \cite{KSO2006}), but with $\alpha_1$ replaced
by $\alpha.$
\end{remark}

\subsection{Coulomb Fluid continued.}
\par
Recall equation $(\ref{c13})$ derived by the Coulomb Fluid method,
$$
\widetilde{X}^4-\al\:\widetilde{X}^3-(2n+\al)t\:\widetilde{X}-t^2=0.
$$
Substituting 
$
t=s/(2n+\al),
$
into quartic, following by $n\to\infty,$ we find,
\bea
\widetilde{X}^3-\alpha\widetilde{X}^2-s=0,\nonumber
\eea
a cubic equation in $\widetilde{X}.$
\par
Let
$$\widetilde{C}=\frac{1}{\widetilde{X}},$$
this cubic becomes,
\bea
\frac{\widetilde{C}^{2}}{s}+\frac{\alpha}{s^{2}}-\frac{1}{s^{2}\widetilde{C}}=0. \nonumber
\eea
\par
Note that this 
is the ``algebraic part" of (\ref{c16}).
\\
\par
Retaining only the real solution of the cubic equation in $\widetilde{C},$ we find,
\bea
\widetilde{C}(s)=-2^{\frac{1}{3}}\:\alpha
\left[27s^{2}+(729s^{4}+108\alpha^{3}s^{3})^{\frac{1}{2}}\right]^{-\frac{1}{3}}+18^{-\frac{1}{3}}s^{-1}
\left[9s^{2}+(81s^{4}+12s^{3}\alpha^{3})^{\frac{1}{2}}\right]^{\frac{1}{3}}. \nonumber
\eea
\par
A Taylor series expansion of $\widetilde{C}(s)$ about $s=0$ gives, for $\alpha>0,$
\bea\label{c45}
\widetilde{C}(s)=\frac{1}{\alpha}-\frac{1}{\alpha^{4}}s+\frac{3}{\alpha^{7}}s^{2}-\frac{12}{\alpha^{10}}s^{3}
+\frac{55}{\alpha^{13}}s^{4}+{\rm O}({s^{5}}).
\eea
\par
For large and positive $s$, one finds
\par
\bea\label{c69}
\widetilde{C}(s)=s^{-\frac{1}{3}}-\frac{\alpha}{3}s^{-\frac{2}{3}}+\frac{\alpha^{3}}{81}s^{-\frac{4}{3}}
+\frac{\alpha^{4}}{243}s^{-\frac{5}{3}}-\frac{4\alpha^{6}}{6561}s^{-\frac{7}{3}}
+{\rm O}({s^{-\frac{8}{3}}}).
\eea
\par
In the next section, we derive the  small $s$ and large $s$ expansion of the $C$ potential, assuming
appropriate formS of the expansions.
\section{Small $s$ and large $s$ behavior of the $C$ potential.}
\par
We consider real-valued solutions of the Painlev\'{e} equations throughout this paper.
\\
We study the solution of the $C$ potential for $s \rightarrow 0^{+}$, by substituting
\bea
C(s)=\sum_{j=0}^{\infty}a_{j}s^{j},\nonumber
\eea
into (\ref{c16}). We find $a_{0}=1/\alpha,$ and $a_1=-1/(\alpha^2(\alpha^2-1)).$
For $n\geq2$, the coefficients satisfy a recurrence relation,
\begin{align*}
\sum_{j=0}^{n}j(j-1)a_{j}a_{n-j}-\sum_{j=0}^{n}j(n-j)a_{j}a_{n-j}+\sum_{j=0}^{n}ja_{j}a_{n-j}
-\sum_{j=0}^{n-1}\sum_{k=0}^{n-1-j}a_{j}a_{k}a_{n-1-j-k}-\alpha{a_{n}}=0.
\end{align*}
\\
A straightforward computation, gives,
\begin{align}\label{c27}
C(s)=&\frac{1}{\alpha}-\frac{1}{\alpha^{2}(\alpha^{2}-1)}s+\frac{3}{\alpha^{3}(\alpha^{2}-1)(\alpha^{2}-4)}s^{2}
-\frac{6(2\alpha^{2}-3)}{\alpha^{4}(\alpha^{2}-1)^{2}(\alpha^{2}-4)(\alpha^{2}-9)}s^{3}\nonumber\\
&+\frac{5(-36+11\alpha^{2})}{\alpha^{5}(\alpha^{2}-1)^{2}(\alpha^{2}-4)(\alpha^{2}-9)(\alpha^{2}-16)}s^{4}\nonumber\\
&+\frac{3(3600-4219\alpha^{2}+1115\alpha^{4}-91\alpha^{6})}
{\alpha^{6}(\alpha^{2}-1)^{3}(\alpha^{2}-4)^{2}(\alpha^{2}-9)(\alpha^{2}-16)(\alpha^{2}-25)}s^{5}+{\rm O}({s^{6}}),
\end{align}
where $\alpha \notin \mathbb{Z}$.
\par
For large and positive $s$, we substitute the asymptotic expansion
\bea
C(s)=\sum_{k=1}^{\infty}b_{k}s^{-\frac{k}{3}}, \nonumber
\eea
into (\ref{c16}), and find ${b_j:\:j=1,2,...,n}$ satisfies
\begin{align*}
\quad &\frac{1}{9}\sum_{j=1}^{n}j(2j+3-n)b_{j}b_{n-j}-\frac{1}{3}\sum_{j=1}^{n}jb_{j}b_{n-j}
-\sum_{j=1}^{n+3}\sum_{k=1}^{n+3-j}b_{j}b_{k}b_{n+3-j-k}-\alpha{b_{n}}=0, \;\;\;n\geq 2,
\end{align*}
$$
{\rm with} \quad 1-b_{1}^{3}=0, \;{\rm and\;}\quad \sum_{j=1}^{4}\sum_{k=1}^{4-j}b_{j}b_{k}b_{4-j-k}+\alpha{b_{1}}=0.
$$
Continuing,
\begin{align}\label{c29}
C(s)=&s^{-\frac{1}{3}}-\frac{\alpha}{3}s^{-\frac{2}{3}}+\frac{\alpha(\alpha^{2}-1)}{81}s^{-\frac{4}{3}}
+\frac{\alpha^{2}(\alpha^{2}-1)}{243}s^{-\frac{5}{3}}+\frac{\alpha(\alpha^{2}-1)}{243}s^{-2}\nonumber\\
&-\frac{2\alpha^{2}(\alpha^{2}-1)(2\alpha^{2}-11)}{6561}s^{-\frac{7}{3}}
-\frac{5\alpha(\alpha^{2}-1)(\alpha^{4}-\alpha^{2}-15)}{19683}s^{-\frac{8}{3}}+{\rm O}({s^{-3}}).
\end{align}
\par
\begin{remark}
Note that for $\al=0,\:\pm 1,$ the
asymptotic expansion terminates. The three relevant algebraic solutions for the $C$ potential are,
\begin{align*}
&C(s)=s^{-\frac{1}{3}},\quad\al=0\nn\\
&C(s)=s^{-1/3}\mp\frac{1}{3}s^{-\frac{2}{3}},\quad\al=\pm1.\nn\\
\end{align*}
\end{remark}
\begin{remark}
The series expansion for the $C$ potential, given in (\ref{c27}), for sufficiently
large $|\alpha|$ is the same as the corresponding series for $\widetilde{C}$ derived by Coulomb Fluid method.
This phenomenon also holds for the large $s$ asymptotic, seen by comparing
(\ref{c29}) with (\ref{c69}).
\end{remark}

\subsection{The behavior of $\cH(s)$ for small $s$ and large $s.$}
Recall {\color{blue}} the $\sigma$-form of our Painlev\'{e} equation,
$$
(s\cH'')^2+4(\cH')^2\left(s\:\cH'-\cH\right)-\left(\al\:\cH'+\frac{1}{2}\right)^{2}=0,
$$
with the initial condition, ${\cal H}(0)=0.$\\
For small $s$, we assume a power series expansion in $s$,
\bea
\cH(s)=\sum_{j=0}^{\infty}d_{j}s^{j}. \nonumber
\eea
Substituting this into the equation satisfied by $\cH(s),$ followed by some computation, gives
\begin{align}\label{c31}
\cH(s)=&-\frac{1}{2\alpha}s+\frac{1}{4\alpha^{2}
(\alpha^{2}-1)}s^{2}-\frac{1}{2\alpha^{3}(\alpha^{2}-1)(\alpha^{2}-4)}s^{3}
+\frac{3(2\alpha^{2}-3)}{4\alpha^{4}(\alpha^{2}-1)^{2}(\alpha^{2}-4)(\alpha^{2}-9)}s^{4}\nn\\
&+\frac{-36+11\alpha^{2}}{2\alpha^{5}(\alpha^{2}-1)^{2}(\alpha^{2}-4)(\alpha^{2}-9)(\alpha^{2}-16)}s^{5}\nn\\
&+\frac{-3600+4219\alpha^{2}-1115\alpha^{4}+91\alpha^{6}}{4\alpha^{6}(\alpha^{2}-1)^{3}
(\alpha^{2}-4)^{2}(\alpha^{2}-9)(\alpha^{2}-16)(\alpha^{2}-25)}s^{6}+{\rm O}({s^{7}}),
\end{align}
where $\alpha \notin \mathbb{Z}$.
\par
For positive and large $s$, we substitute the  expansion
\bea
\cH(s)=s^{\frac{2}{3}}\left(\sum_{j=0}^{\infty}\eta_{j}s^{-\frac{j}{3}}\right), \nonumber
\eea
into the $\sigma$-form equation. After some computations, the expansions of ${\cal H}(s),$ for large and positive $s$ read,
\begin{align}\label{c33}
\cH(s)=&-\frac{3}{4}s^{\frac{2}{3}}+\frac{\alpha}{2}s^{\frac{1}{3}}+\frac{1-6\alpha^{2}}{36}+\frac{\alpha(\alpha^{2}-1)}{54}s^{-\frac{1}{3}}
+\frac{\alpha^{2}(\alpha^{2}-1)}{324}s^{-\frac{2}{3}}
+\frac{\alpha(\alpha^{2}-1)}{486}s^{-1}\nonumber\\
&-\frac{\alpha^{2}(\alpha^2-1)(2\alpha^{2}-11)}{8748}s^{-\frac{4}{3}}-\frac{\alpha(\alpha^{6}-2\alpha^{4}
-14\alpha^{2}+15)}{13122}s^{-\frac{5}{3}}\nonumber\\
&-\frac{8\alpha^{6}-41\alpha^{4}+33\alpha^{2}}{26244}s^{-2}+{\rm O}({s^{-\frac{7}{3}}}).
\end{align}
The next Theorem gives the small $s$ and large $s$ expansions of $\Delta(s,\alpha).$
\begin{theorem}
In the double scaling $n\rightarrow\infty$, $t\rightarrow 0$, such that $s:=(2n+1+\al)t$ and $s\in(0,\infty)$,
the expressions of $\Delta(s,\alpha),$  for small $s$ and large $s$ are as follows:
\par
For small $s$,
\begin{align}\label{c41}
&\Delta(s,\al)=\exp\left[-\frac{1}{2\alpha}s+\frac{1}{8\alpha^{2}(\alpha^{2}-1)}s^{2}
-\frac{1}{6\alpha^{3}(\alpha^{2}-1)(\alpha^{2}-4)}s^{3}\right.\nonumber\\
&+\frac{3(2\alpha^{2}-3)}{16\alpha^{4}(\alpha^{2}-1)^{2}(\alpha^{2}-4)(\alpha^{2}-9)}s^{4}
+\frac{36-11\alpha^{2}}{10\alpha^{5}(\alpha^{2}-1)^{2}(\alpha^{2}-4)(\alpha^{2}-9)(\alpha^{2}-16)}s^{5}\nonumber\\
&\left.+\frac{91\alpha^{6}-1115\alpha^{4}+4219\alpha^{2}-3600}{24\alpha^{6}
(\alpha^{2}-4)^{2}(\alpha^{2}-1)^{3}(\alpha^{2}-9)(\alpha^{2}-16)(\alpha^{2}-25)}s^{6}+{\rm O}(s^{7})\right],
\end{align}
where $\alpha \notin \mathbb{Z}$.
\par
For large $s$,
\begin{align}\label{b30}
\Delta(s,\al)=&\exp\left[c_{1}-\frac{9}{8}s^{\frac{2}{3}}+\frac{3\alpha}{2}s^{\frac{1}{3}}
+\frac{1-6\alpha^{2}}{36}\ln{s}+\frac{\alpha(1-\alpha^{2})}{18}s^{-\frac{1}{3}}
+\frac{\alpha^{2}(1-\alpha^{2})}{216}s^{-\frac{2}{3}}\right.\nonumber\\
&+\frac{\alpha(1-\alpha^{2})}{486}s^{-1}
+\frac{\alpha^{2}(2\alpha^{4}-13\alpha^{2}+11)}{11664}s^{-\frac{4}{3}}
+\frac{\alpha(\alpha^{6}-2\alpha^{4}-14\alpha^{2}+15)}{21870}s^{-\frac{5}{3}}\nonumber\\
&\left.+{\rm O}(s^{-3})\right],
\end{align}
 where $c_{1}=c_{1}(\alpha)$ is a constant independent of $s$.
\end{theorem}
\begin{proof}
Recall (\ref{cc1}),
\begin{equation*}
\ln{\Delta(s,\alpha)}=\int_{0}^{s}\left\{\frac{1}{4}\left(\frac{\xi}{C(\xi)}\:\frac{dC(\xi)}{d\xi}\right)^2-
\frac{\xi\:C(\xi)}{2}-\frac{1}{4}\left(\frac{1}{C(\xi)}-\al\right)^2\right\}\frac{d\xi}{\xi}.
\end{equation*}
 Substituting (\ref{c27}) and (\ref{c29}) into the above formula, then the expansions of $\Delta(s,\al)$ for small $s,$ (\ref{c41}) and for large $s,$ (\ref{b30}) follow immediately.\\
The results obtained coincide with the integration of
$$
{\cal H}(s)=s\frac{d}{ds}\ln\Delta(s,\al), \quad {\cal H}(0)=0,
$$
with the expansions of ${\cal H}(s)$ for small $s,$ (\ref{c31}) and for large $s,$ (\ref{c33}).
\end{proof}
At the end of this section we compute the large $n$ behavior of $P_n(0;t,\alpha),$ namely the evaluation of the orthogonal polynomials
at the origin, from the fact that,
\begin{equation*}
(-1)^nP_n(0; t,\alpha) = \frac{D_{n}(t, \alpha + 1)}{D_{n}(t, \alpha)}.
\end{equation*}
Note the exact evaluation of $(-1)^nP_n(0;0,\alpha)$ in (\ref{gamma}).
\begin{cor}
Under 
double scaling, 
\begin{align}\label{c67}
\lim\limits_{n \rightarrow \infty}\frac{(-1)^nP_n\left(0; \frac{s}{2n+\alpha+1},\alpha\right)}{(-1)^nP_n\left(0;0,\alpha\right)}&=
\frac{\Delta(s,\alpha+1)}{\Delta(s,\alpha)}
=\exp\left(c_{2}+\frac{3}{2}s^{\frac{1}{3}}-\frac{1+2\alpha}{6}\ln{s}-\frac{\alpha(\alpha+1)}{6}s^{-\frac{1}{3}}\right.\nonumber\\
&\left.-\frac{\alpha(\alpha+1)(2\alpha+1)}{108}s^{-\frac{2}{3}}-\frac{\alpha(\alpha+1)}{162}s^{-1}\right.\nonumber\\
&\left.+\frac{\alpha(\alpha+1)(2\alpha+1)(\alpha^{2}+\alpha-3)}{1944}s^{-\frac{4}{3}}+{\rm O}(s^{-\frac{5}{3}})\right),
\end{align}
where $c_{2}=c_{2}(\alpha)$ is a constant independent of $s$.
\end{cor}
\begin{proof}
From the fact that
\begin{align*}
\lim\limits_{n \rightarrow \infty}\frac{(-1)^nP_n\left(0; \frac{s}{2n+\alpha+1},\alpha\right)}
{(-1)^nP_n\left(0;0,\alpha\right)}&=\lim_{n \rightarrow \infty}\frac{D_{n}(s/(2n+\alpha+1), \alpha + 1)}
{D_{n}(0,\alpha + 1)}\frac{D_{n}(0,\alpha)}{D_{n}(s/(2n+\alpha+1), \alpha)}\\
&=\frac{\Delta(s,\alpha+1)}{\Delta(s,\alpha)}.
\end{align*}
Together with the expression of $\Delta(s,\alpha+1)$ and $\Delta(s,\alpha)$ given by
(\ref{b30}); the equation (\ref{c67}) is obtained. Furthermore, there is
a relation between $c_{1}(\alpha)$ and $c_{2}(\alpha)$, namely,
\begin{equation}\label{rr1}
c_{2}(\alpha)=c_{1}(\alpha+1)-c_{1}(\alpha).
\end{equation}
\end{proof}
In the next section a computation produces the constant $c_2(\alpha).$
\section{The asymptotic of $P_{n}(0; t, \alpha)$.}
In this section, we evaluate $P_n(0; t, \alpha)$, for large $n,$ and give a derivation of $c_{2}(\alpha)$.
\\
First we state a result regarding the large $n$ behaviour of the orthogonal polynomials $P_{n}(z; t,\alpha)$, for $z\notin{[a,b]}.$
In our problem, the potential, ${\rm v}(x)=-\alpha\ln{x}+x+\frac{t}{x}, x \geq 0, t>0$ and $\alpha>0$, satisfy
the convexity condition\cite{YN1998}.

The formulas below are valid for $z\notin [a,b],$ and large $n.$
\par
For large $n,$ $P_{n}(z)$ can be computed as
\bea\label{c52}
P_{n}(z) \sim \exp[-S_{1}(z)-S_{2}(z)], \quad {\rm where} \quad z \notin [a,b],
    \eea
and  $S_{1}(z)$ and $S_{2}(z)$ are given by ($(4.6)$ and $(4.7)$, \cite{YN1998}).

These are
\begin{align*}
S_{1}(z)=\frac{1}{4}\ln\left[\frac{16(z-a)(z-b)}{(b-a)^{2}}\left(\frac{\sqrt{z-a}-\sqrt{z-b}}{\sqrt{z-a}+\sqrt{z-b}}\right)^{2}\right],
\quad z \notin [a,b],
\end{align*}
and
\begin{align}\label{c54}
S_{2}(z)=-n\ln&\left(\frac{\sqrt{z-a}+\sqrt{z-b}}{2}\right)^{2}\nonumber\\
&+\frac{1}{2\pi}\int_{a}^{b}\frac{{\rm v}(x)}{\sqrt{(b-x)(x-a)}}\left[\frac{\sqrt{(z-a)(z-b)}}{x-z}+1\right]dx, \quad z \notin [a,b].
\end{align}
\\
We mention here an equivalent representation of $S_1(z),\;\;z\notin[a,b]$,
\bea\label{c53}
\exp\left(-S_{1}(z)\right)=
\frac{1}{2}\left[\left(\frac{z-b}{z-a}\right)^{\frac{1}{4}}+\left(\frac{z-a}{z-b}\right)^{\frac{1}{4}}\right].
\eea
The next theorem gives an evaluation of $P_n(0;t,\alpha)$ for large $n.$
\\
\begin{theorem}\label{s1s2}
If ${\rm v}(x)=-\alpha\ln{x}+x+\frac{t}{x}, x \geq 0, t>0$ and $\alpha>0$, the evaluations
at $z=0$  of $S_{1}(z;t,\alpha)$, $S_{2}(z;t,\alpha)$,
and $P_n(z;t,\alpha)$ are given by
\bea\label{c55}
\exp[-S_{1}(0; t, \alpha)] \sim (2t)^{-\frac{1}{6}}n^{\frac{1}{3}},
\eea
\bea\label{c56}
(-1)^{n}\exp[-S_{2}(0; t, \alpha)] \sim n^{n}(2t)^{-\frac{\alpha}{3}}
\exp\left(-n+3\cdot2^{-\frac{2}{3}}n^{\frac{1}{3}}t^{\frac{1}{3}}+\frac{2\alpha}{3}\ln{n}\right),
\eea
and
\begin{align}\label{c57}
(-1)^{n}P_{n}(0; t, \alpha) &\sim (-1)^{n}\exp[-S_{1}(0; t, \alpha)-S_{2}(0; t, \alpha)]\nonumber\\
 &\sim n^{n}(2t)^{-\frac{1}{6}-\frac{\alpha}{3}}\exp\left(-n+3\cdot2^{-\frac{2}{3}}n^{\frac{1}{3}}t^{\frac{1}{3}}
 +\frac{1}{3}(1+2\alpha)\ln{n}\right),
\end{align}
furthermore, the asymptotic estimates are uniform with respect to\\
 $t\in(0, a_{0}], 0<a_{0}<\infty $,
$\alpha>0$, $n\rightarrow\infty$ and $nt\rightarrow\infty$.
\end{theorem}
\begin{proof}
Recall the quartic equation satisfied by $\widetilde{X},$
$$
\widetilde{X}^4-\al\:\widetilde{X}^3-(2n+\al)t\:\widetilde{X}-t^2=0,
$$
where $\widetilde{X}=\sqrt{ab}$.
\par
Let $\widetilde{n}=2n+\alpha$. The relevant solution of the quartic, reads, for large $\widetilde{n},$
\bea
\frac{1}{\widetilde{X}}=(\widetilde{n}t)^{-\frac{1}{3}}-\frac{\alpha}{3}(\widetilde{n}t)^{-\frac{2}{3}}
+\frac{\alpha^{3}}{81}(\widetilde{n}t)^{-\frac{4}{3}}+\left(\frac{\alpha^{4}}{243}-\frac{t^2}{3}\right)(\widetilde{n}t)^{-\frac{5}{3}}
+{\rm O}\left((\widetilde{n}t)^{-2}\right),\nonumber
\eea
where $t\in(0, a_{0}], 0<a_{0}<\infty $, $\alpha>0$, $n\rightarrow\infty$.
\par
From (\ref{c11}) and (\ref{c12}), we see that
\begin{align*}
a=\widetilde{n}+\frac{t}{\widetilde{X}}-\sqrt{\left(\widetilde{n}+\frac{t}{\widetilde{X}}\right)^{2}-\widetilde{X}^{2}}=
\frac{t^{\frac{2}{3}}}{2\widetilde{n}^{\frac{1}{3}}}+\frac{\alpha{t^{\frac{1}{3}}}}{3\widetilde{n}^{\frac{2}{3}}}
+\frac{\alpha^{2}}{6\widetilde{n}}
+\frac{5\alpha^{3}}{81t^{\frac{1}{3}}\widetilde{n}^{\frac{4}{3}}}
+{\rm O}(\widetilde{n}^{-\frac{5}{3}}),
\end{align*}
and
\begin{align*}
b=\widetilde{n}+\frac{t}{\widetilde{X}}+\sqrt{\left(\widetilde{n}+\frac{t}{\widetilde{X}}\right)^{2}
-\widetilde{X}^{2}}=2\widetilde{n}+\frac{3t^{\frac{2}{3}}}{2\widetilde{n}^{\frac{1}{3}}}
-\frac{\alpha{t^{\frac{1}{3}}}}{\widetilde{n}^{\frac{2}{3}}}-\frac{\alpha^{2}}{6\widetilde{n}}
-\frac{\alpha^{3}}{27t^{\frac{1}{3}}\widetilde{n}^{\frac{4}{3}}}+{\rm O}(\widetilde{n}^{-\frac{5}{3}}),
\end{align*}
where $t\in(0, a_{0}],$ $0<a_{0}<\infty$, $\alpha>0$, $n\rightarrow\infty$.
\par
Hence
\bea
ab=(\widetilde{n}t)^{\frac{2}{3}}+\frac{2}{3}\alpha(\widetilde{n}t)^{\frac{1}{3}}+ {\rm O}(1),
\quad {\rm which\:implies,} \quad \sqrt{ab}=(\widetilde{n}t)^{\frac{1}{3}}+\frac{\alpha}{3}+{\rm O}\left(n^{-\frac{1}{3}}\right). \nonumber
\eea
\par
From the expression for $\exp(-S_{1}(z))$, see (\ref{c53}), we find,
\bea\label{c63}
\exp[-S_{1}(0; t, \alpha)]\sim\frac{1}{2}\left(\frac{b}{a}\right)^{\frac{1}{4}}\sim 2^{-\frac{1}{6}}n^{\frac{1}{3}}t^{-\frac{1}{6}}.
\eea
We now evaluate $S_{2}(0;t,\alpha)$ by setting $z=0$ in (\ref{c54}). 
With the aid of the integral identities in Appendix $A$, followed by some computations, we find,
\begin{align*}
\exp[-S_{2}(0; t, \alpha)]&=
(-1)^{n}\left(n+\frac{\alpha}{2}
+\frac{t}{2\sqrt{ab}}+\frac{\sqrt{ab}}{2}\right)^{n}\left(\frac{n}{\sqrt{ab}}+\frac{\alpha}{2\sqrt{ab}}
+\frac{t}{2ab}+\frac{1}{2}\right)^{\alpha}\nonumber\\
&\times \exp\left[-n-\frac{\alpha}{2}-\frac{t}{\sqrt{ab}}+\frac{\sqrt{ab}}{2}+\left(n+\frac{\alpha}{2}\right)\frac{t}{ab}
+\frac{t^{2}}{2(ab)^{\frac{3}{2}}}\right].
\end{align*}
Now since $\sqrt{ab}\sim (\widetilde{n}t)^{\frac{1}{3}}+\frac{\alpha}{3}$,
the above becomes (\ref{c56}). With the expressions for
${\rm exp}(-S_{1}(0; t, \alpha))$  and ${\rm exp}(-S_{2}(0; t, \alpha)),$ the asymptotic
 estimations for $P_{n}(0; t, \alpha),$ namely, (\ref{c57}) follows immediately.
\end{proof}

\begin{remark}
For convenience, we rewrite the asymptotic estimation of $(-1)^{n}P_{n}(0; t, \alpha)$, (\ref{c57}), as
\begin{align}\label{c68}
(-1)^{n}P_{n}(0; t, \alpha) &\sim (-1)^{n}\exp[-S_{1}(0; t, \alpha)-S_{2}(0; t, \alpha)]\nonumber\\
 &\sim n^{n}(2t)^{-\frac{1}{6}-\frac{\alpha}{3}}\exp\left(-n+3\cdot2^{-\frac{2}{3}}n^{\frac{1}{3}}t^{\frac{1}{3}}
 +\frac{1}{3}(1+2\alpha)\ln{n}\right)\nonumber\\
 &=n^{n+\alpha+\frac{1}{2}}e^{-n}\exp\left(\frac{3}{2}(2nt)^{\frac{1}{3}}-\frac{1+2\alpha}{6}\ln(2nt)\right)\nonumber\\
 &\sim \Gamma(n+1+\alpha)\exp\left(-\frac{1}{2}\ln(2\pi)+\frac{3}{2}(2nt)^{\frac{1}{3}}-\frac{1+2\alpha}{6}\ln(2nt)\right)\nonumber\\
 &= \frac{\Gamma(n+1+\alpha)}{\Gamma(1+\alpha)}\exp\left(\ln\left(\frac{\Gamma(1+\alpha)}{\sqrt{2\pi}}\right)
 +\frac{3}{2}(2nt)^{\frac{1}{3}}-\frac{1+2\alpha}{6}\ln(2nt)\right),
 \end{align}
and hence (noting the exact evaluation---(2.6)),
 $$
 \frac{P_n(0; t,\alpha)}{P_n(0; 0,\alpha)}\sim
 {\rm exp}\left(c_2(\alpha)+\frac{3}{2}\:(2nt)^{1/3}-\frac{1+2\alpha}{6}\:\ln(2nt)\right),
 $$
 from which  $c_2(\alpha)$ is identified to be
 $$
 \ln\left(\frac{\Gamma(1+\alpha)}{\sqrt{2\pi}}\right).
 $$
In the pan-ultimate step, leading to (\ref{c68}),
$$
{\sqrt{2\pi}}\:{\rm e}^{(n+\alpha+1/2)\ln{n}-n}
$$
is replaced by
$$
\Gamma(n+1+\alpha).
$$
The singular perturbation, obtained through the multiplication of ${\rm e}^{-t/x}$ on Laguerre weight,
$x^{\alpha}{\rm e}^{-x}$  causes a large distortion, in such a way that,
$$
\frac{P_n(0; t,\alpha)}{P_n(0; 0,\alpha)}
\sim {\rm exp}\left(\frac{3}{2}\;(2nt)^{1/3} -\frac{1+2\alpha}{6}\:\ln(2nt)+c_2(\alpha)\right).
$$
Identifying $2nt$ by $s$, then (\ref{c68}) becomes (\ref{c67}).
\par
Hence, the equation (\ref{rr1}) reads,
$$
c_{1}(\alpha+1)-c_{1}(\alpha)=c_{2}(\alpha)=\ln\left(\frac{\Gamma(1+\alpha)}{\sqrt{2\pi}}\right),\quad \alpha>0.
$$
Noting the properties of the Barnes-$G$ function, gives,
$$
c_1(\alpha)=\ln\:\frac{G(\alpha+1)}{(2\pi)^{\alpha/2}}.
$$
\end{remark}
However, a rigorous determination of $c_{1}(\alpha)$ remains open.

\section{Conclusion}

The (finite $n$) Hankel determinant generated by a singularly deformed
Laguerre weight, $x^{\alpha}{\rm exp}(-x - t/x),\;\;\alpha>-1,\;\;t\geq 0,\;\;x\geq 0,$
was shown to be intimately related to a particular (finite $n$) ${\rm P_{III}}.$
In the double scaling scheme considered in this paper, the ${\it infinite Hankel determinant,}$  has an integral representation in terms of
a $C$ potential, which satisfies a second order non-linear ode. Up to minor changes of variable, this is equivalent
to a ${\rm P_{III}}$ with smaller number of parameters. Asymptotic expansion of
the {\it scaled} Hankel determinant are found, through the relevant Painlev\'e equations.
\\
\\
{\bf Appendix A: Some Integration Identities.}
\par
The integrals listed below, valid for $0<a<b,$  can be found in the\cite{GR2007}, and in \cite{YHM2013} and \cite{YM2012}:
$$
\int_{a}^{b}\frac{dx}{\sqrt{(b-x)(x-a)}}=\pi. \eqno{(A1)}
$$

$$
\int_{a}^{b}\frac{x\:dx}{\sqrt{(b-x)(x-a)}}=\frac{(a+b)\pi}{2}. \eqno{(A2)}
$$

$$
\int_{a}^{b}\frac{dx}{x\sqrt{(b-x)(x-a)}}=\frac{\pi}{\sqrt{ab}}. \eqno{(A3)}
$$

$$
\int_{a}^{b}\frac{dx}{x^2\sqrt{(b-x)(x-a)}}=\frac{(a+b)\pi}{2(ab)^{\frac{3}{2}}}. \eqno{(A4)}
$$

$$
\int_{a}^{b}\frac{\ln{x}}{\sqrt{(b-x)(x-a)}}dx=2\pi\ln\left(\frac{\sqrt{a}+\sqrt{b}}{2}\right). \eqno{(A5)}
$$

$$
\int_{a}^{b}\frac{\ln{x}}{x\sqrt{(b-x)(x-a)}}dx=\frac{2\pi}{\sqrt{ab}}\ln\frac{2\sqrt{ab}}{\sqrt{a}+\sqrt{b}}. \eqno{(A6)}
$$
\\
{\bf Acknowledgement.}
\par
We would like to thank the Macau Science and Technology Development Fund for generous support: FDCT 077/2012/A3.

\end{document}